\documentclass[preprint,11pt]{elsarticle}

\usepackage{amssymb}
\usepackage{amsmath}
\usepackage{amsfonts} 
\usepackage{amsbsy}
\usepackage{amscd}
\usepackage{amsthm}
\usepackage{geometry}                
\usepackage{graphicx}
 \usepackage{array,multirow}
\usepackage{amssymb}
\usepackage{epstopdf}
\usepackage{mathtools}
\usepackage{mathrsfs}
\usepackage[capitalize]{cleveref}

\usepackage{pgfplots}
\pgfplotsset{compat=1.15}
\usepackage{float} 
\usepackage[framemethod=tikz]{mdframed} 
\usepackage{circuitikz}
\usetikzlibrary{decorations.pathmorphing, patterns,shapes}

\usetikzlibrary{arrows}
\usetikzlibrary{patterns}
\usepackage{xcolor}
\usepackage{caption}
\usepackage{subcaption}

\newcommand{\Hidden}[1]{}

\def\big{\bigskip}

\newtheorem{theorem}{Theorem}
\newtheorem{thm}[theorem]{Theorem}
\newtheorem{lemma}[theorem]{Lemma}
\newtheorem{prop}[theorem]{Proposition}
\newtheorem{corollary}[theorem]{Corollary}
\newtheorem{conj}[theorem]{Conjecture}
\newtheorem*{ktzordering}{Theorem \ref{conj:Katzordering}}

\newtheorem*{cnjroots}{Conjecture \ref{conj:Katzroots}}
\newtheorem*{ktzcycle}{Theorem \ref{thm:cycleordering}}
\newtheorem*{cycleconvth}{Theorem \ref{cycleconv}}
\newtheorem*{pathconvth}{Theorem \ref{pathcoev}}
\newtheorem*{proprootfive}{Proposition \ref{prop:Katroot5}}
\newtheorem*{lemmacycle1}{Lemma \ref{lem:cycle-det}}

\theoremstyle{definition}

\newtheorem{defn}[theorem]{Definition}

\theoremstyle{remark}

\newcommand{\aef}[1]{{\color{blue} \sf AEF: [#1]}}

\begin{document}
\begin{frontmatter}
\title{Katz similarity index comparisons}

\author[1]{Emily J. Evans\fnref{fn1}}
\ead{ejevans@mathematics.byu.edu}
\author[2]{Amanda E. Francis\fnref{fn1}}
\ead{aefr@umich.edu}
\author[3]{Rebecca D. Jones}
\ead{rebecca.d.jones@pnnl.gov}
\address[1]{Department of Mathematics, Brigham Young University, Provo, UT 84602, USA}
\address[2]{Mathematical Reviews, American Mathematical Society, Ann Arbor, MI 48103, USA}
\address[3]{Pacific Northwest National Labs, Richland, WA 99352, USA}

\fntext[fn1]{This material is based upon work supported by the National Security Agency
under Grant No. H98230-19-1-0119, while the authors were in residence at the Simons Laufer Mathematical Sciences Institute in Berkeley, California, during the summer of 2019.} 

\begin{abstract}{
In this paper we compare the Katz similarity index to two other node similarity metrics, the standard distance and the resistance distance, for both path and cycle graphs.  We consider how  Katz similarity  varies as the parameter $\alpha$ and the size of the graph vary. We also characterize when the three metrics give rise to different orderings of vertex pairs.  In particular, we find that for all admissible values of $\alpha$ the Katz similarity index, the resistance distance, and the distance give the same ordering for node pairs in a cycle graph of arbitrary length, but the same is not true for the path graph.}
\end{abstract}


\end{frontmatter}

\section{Introduction}

One problem of particular interest in network science is the identification of missing links and the prediction of new links.
There are many ways one could do this, but it typically involves calculating a similarity measure between pairs of nodes and predicting the $k$ (non-edge) pairs with the maximum likelihood. Methods of identifying new edges can be split into local similarity methods which includes common neighbors, Jaccard’s coefficient, preferential attachment, resource allocation and Adamic-Adar; distance (path) methods including Katz similarity index, resistance distance and rooted PageRank; and kernel-based similarity measures such as commute time kernel and exponential diffusion kernel; and graph neural networks~\cite{fouss_algorithms_2016-1,pachev_fast_2018,wu_graph_2022}.

Although there have been a large number of studies touting the advantages of one method over another for particular types of problems, the literature lacks an in-depth analysis of when these algorithms return the same predicted links and when they do not.  This paper fills this gap by comparing methods of determining the similarity between node pairs: the Katz similarity index, standard distance, and resistance distance. 

Recall that the resistance distance, also called effective resistance, is calculated by considering the graph $G$ as an electrical circuit with each edge being replaced by a resistor.  The resistance distance is then defined as the resistance between two nodes on that circuit.  Although there are various ways of calculating resistance distance (see for example~\cite{evans2021algorithmic}) a common way uses the Moore-Penrose pseudoinverse of the Laplacian matrix of $G$
\[R_G(i,j) = (L_G^{\dagger})_{ii}+(L_G^{\dagger})_{jj}-2(L_G^{\dagger})_{ij}.\]

In contrast,  the Katz similarity index gives a weighted sum of all paths of various lengths between the two nodes.  More formally,
\begin{defn}
Let $G$ be a simple graph with adjacency matrix $A$, and 
let $0 < \alpha <\frac{ 1}{\rho(A)}$. Then the  {\it Katz matrix of $G$ at $\alpha$} is given by 
\[
K_G(\alpha) = \alpha A + \alpha^2 A^2 + \alpha^3 A^3 +\cdots = \sum_{t = 1}^\infty \alpha^t A^t =  (I - \alpha A)^{-1} - I 
\]
and the {\it Katz similarity index} at $\alpha$ between vertices $i$ and $j$ of $G$ is given by $K_G(\alpha)_{i,j}.$

\end{defn}

%In particular, given a graph $G$ with adjacency matrix $A$ we define the Katz matrix, which encapsulates distances between any pair of vertices as $K(\alpha)$, where \[K(\alpha) = \sum_{i=1}^{\infty}\alpha^iA^i.\] \aef{How is this different than Definition 1? } 

In the link prediction problem,  vertex pairs with the maximal Katz similarity index or the minimal effective resistance are considered. Katz similarity index is affected by the decay parameter $\alpha$; thus,  changing $\alpha$ may yield a different pair ordering.  The purpose of this work is to explore the effect of changing $\alpha$ on the ordering of vertex pairs for two families of simple graphs: paths and cycles. In Figure~\ref{fig:scatter_path_cycle} we demonstrate the effect of changing $\alpha$ on the Katz similarity index in the case of paths and cycles.  In the upper right panel we display the case when $\alpha =0.46$ for the path graph.  Observe that vertices that are $k$ apart have lower Katz scores than vertices that are $k+1$ apart resulting in a different ordering.  Whereas in the lower panels, corresponding to the cycle graph, the same ordering is observed regardless of the value of $\alpha$.  In particular we have the following results for cycles and paths:

\begin{figure}[h]
\begin{picture}(400,240)
\put(-10,10){\includegraphics[width=1\textwidth]{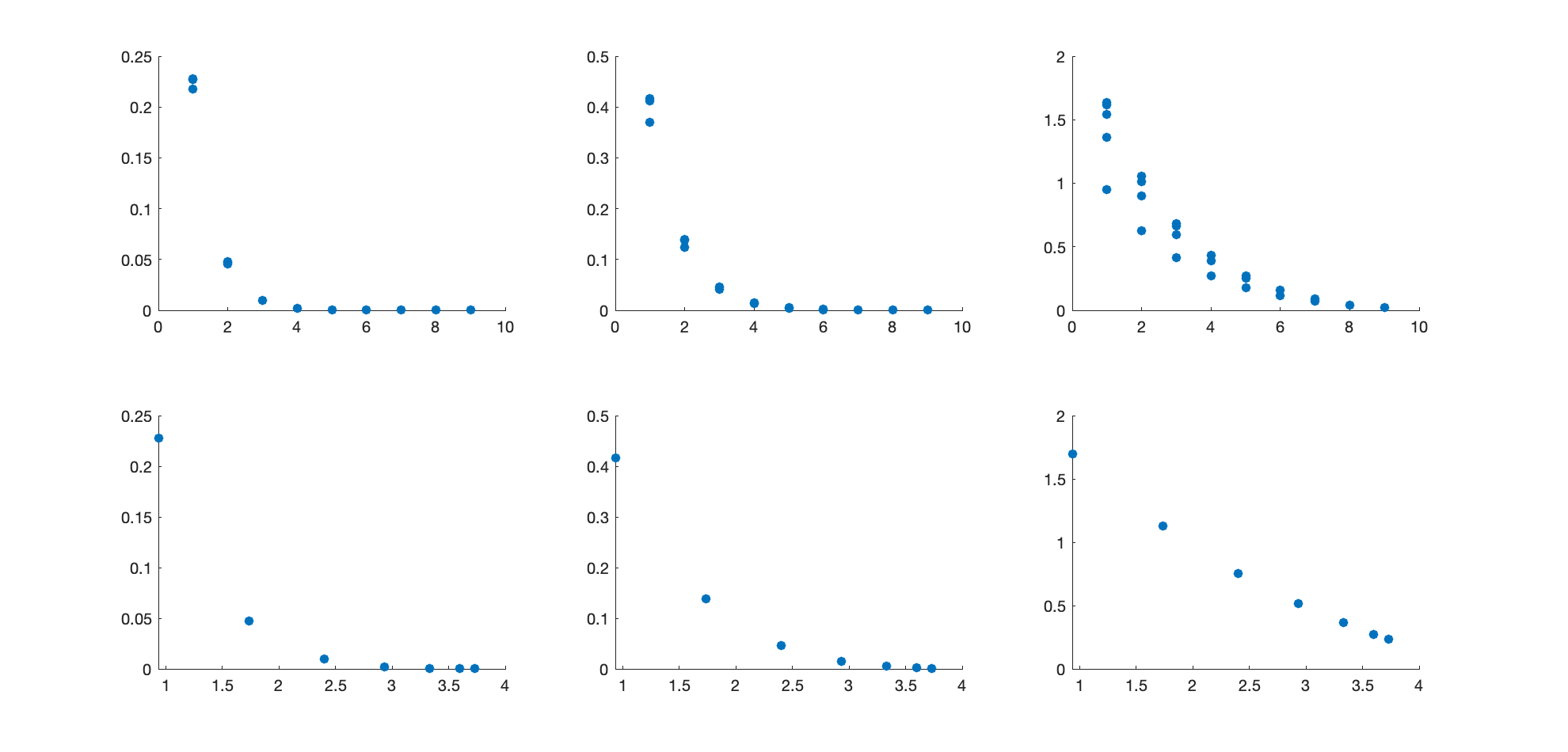}}
\put(0,160){{\rotatebox[origin=c]{90}{$P_{10}$}}}
\put(0,60){{\rotatebox[origin=c]{90}{$C_{15}$}}}
\put(70,0){{$\alpha = 0.2$}}
\put(190,0){{$\alpha = 0.3$}}
\put(310,0){{$\alpha = 0.46$}}
\end{picture}
\caption{Katz similarity scores versus effective resistance for pairs of vertices in the path graph $P_{10}$ on 10 vertices and the cycle graph $C_{15}$ on 15 vertices with varying values of $\alpha$.  Observe that when $\alpha =0.46$ some vertex pairs in $P_{10}$ that are $k$ apart have lower Katz scores than some pairs that are $k+1$ apart, resulting in a different ordering.}\label{fig:scatter_path_cycle}
\end{figure}
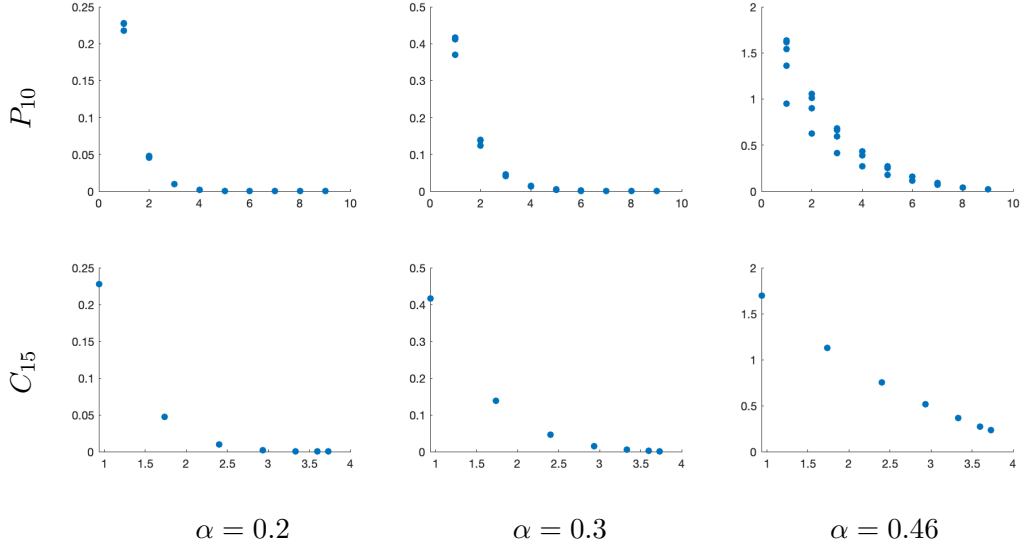
\begin{theorem}\label{thm:cycleordering}Let $C_n$ be a cycle graph on $n$ vertices, with associated adjacency matrix $A_n$. For every valid value of $\alpha$ (i.e., $0<\alpha < 1/\rho(A_n)$)  the effective resistance, distance, and Katz similarity rank all node pairs of $G$ in the same order.  That is, 
\[
K_{C_n}(\alpha)_{i,j} < K_{C_n}(\alpha)_{i',j'}
%\text{ if and only if }
\Leftrightarrow
R_{C_n}(i,j)>R_{C_n}(i',j') 
%\text{ if and only if }
\Leftrightarrow
d_{C_n}(i,j)>d_{C_n}(i',j')
\]
for all $i,i',j,j' \in [n]$.
\end{theorem}

\begin{theorem}\label{conj:Katzordering} Let $P_n$ be a path graph on $n$ of vertices. The resistance distance and usual graph distance are equal, thus they always rank all node pairs equivalently.   If $\alpha < 1/{\sqrt{5}}$, then effective resistance (equivalently, distance) and Katz similarity rank all node pairs of $G$ in the same order. That is, 
\[
K_{P_n}(\alpha)_{i,j} < K_{P_n}(\alpha)_{i',j'}
%\text{ if and only if }
\Leftrightarrow
R_{P_n}(i,j)>R_{P_n}(i',j')
\]
for $\alpha< 1/\sqrt{5}$ and for all $i,i',j,j' \in [n]$.
\end{theorem}
In fact, for each $n$ there is a {\it cut-off} value $\alpha_n$ for which the Katz similarity index and resistance distance rankings agree 
for $\alpha< \alpha_n$ and disagree
for $\alpha>\alpha_n$ (see Theorem~\ref{thm:rootsexist}).  
Computational evidence for the following conjecture will also be presented: 

\begin{conj}\label{conj:Katzroots} 
The cut-off values $\alpha_n$ converge to $\frac{1}{\sqrt{5}}$ from above. 
\end{conj}

Another question of interest related to Katz similarity, is how Katz similarity index is affected by the size of a graph, i.e., to what extent is this a measure of local versus global behavior. We show the following two theorems. 

\begin{theorem}\label{pathcoev}
    $K_{P_n}(\alpha)_{i,j}$ converges as $n$ goes to $\infty$. 
\end{theorem}
\begin{theorem}\label{cycleconv}
$K_{C_n}(\alpha)_{i,j}$ converges as $n$ goes to $\infty$. 
\end{theorem}

The calculation of Katz similarity index on paths has perhaps an unexpected connection to a generalization of Dyck paths.  Recall, a \emph{Dyck Path} of length $2n$ is a staircase path on a lattice from the point $(0,0)$ to the point $(n,n)$ with the restriction that the path does not go above the diagonal $x=y$.  
%Alternatively a Dyck path of length $2n$ can be though of as a one-dimensional walk taking steps in $\{1,-1\}$ with the restriction that the walk starts and ends at 0, and that each point along the walk is nonnegative.  It is well known that the number of Dyck paths of length $2n$ is given by the $n$th Catalan number 
%\[C_n =\frac{1}{n+1}{2n \choose n}.\]
The original definition of a Dyck path has been generalized in a number of different ways.  For example the restriction on walks not crossing the $x$-axis can be relaxed, as in the so-called GranDyck paths or Dyck bridges (see \cite{Comtet, Pergola, Ferrari, dastidar}); the ending point $(n,n)$ can be replaced by a more general point $(a,b)$ (see \cite{Imaoka}). Our work on the Katz similarity index, in relation to walks on paths of fixed length gives both new results and new proofs.  This would take us too far afield of the results presented in this paper but the interested reader is refered to~\cite{DyckPaths}.

In Section~\ref{sec:KatzDistance} we derive exact formulas for the Katz similarity index on path and cycle graphs. In Section~\ref{sec:Convergence} we consider how the Katz similarity index is affected by the size of the graph and show that both $K_{P_n}(\alpha)_{i,j}$ and $K_{C_n}(\alpha)_{i,j}$ converge as $n$ goes to $\infty$.  
%Prior to discussing when the Katz similarity index and resistance distance return the same ordering we draw connections between Katz similarity index and Dyck bridges in Section~\ref{sec:Dyck}.  
Finally in Section~\ref{sec:relationship} we show Theorems~\ref{thm:cycleordering} and~\ref{conj:Katzordering}, and present a conjecture strengthening Theorem~\ref{conj:Katzordering}.

\section{Katz similarity index in Path and Cycle Graphs}\label{sec:KatzDistance}
In this section we give explicit formulas for the Katz similarity index between any two vertices on a cycle or path graph. These formulas will be used in Section~\ref{sec:Convergence} to show the convergence of the Katz similarity index as the number of vertices goes to infinity and in Section~\ref{sec:relationship} as we consider the relationship between Katz similarity index and the more well-known resistance distance.

First, we introduce a combinatorial function that will be useful to us throughout this paper, and which is interesting in its own right. 

\begin{defn}\label{def:dn}
    \[
d_i(\alpha) := \sum_{m = 0}^{\lfloor{i/2}\rfloor} (-1)^m { i-m \choose m} \alpha^{2m}.
\]
\end{defn}
Next, we collect a handful of useful equalities and relations describing some of the surprising Fibonacci-like behavior of this function, including an alternative recursive definition for it.    The proofs of the various parts of the theorem can be found in the Appendix. 
% \[\mathbb{F}_i = \begin{cases}
%         L_i & i \text{ odd}\\
%         F_i & i \text{ even}
%     \end{cases}, \qquad 

\begin{theorem}\label{thm:alldfacts}
For $d_i(\alpha)$ as defined in Definition~\ref{def:dn}, the following hold true:
\begin{enumerate}
    \item \label{dfactrecursion} 
    The sequence of functions $d_n(\alpha)$ can be defined recursively: 
    $$d_0(\alpha) = d_{1}(\alpha) = 1. \quad \text{For }i\geq 2, \ d_n(\alpha) = d_{n-1}(\alpha) - \alpha^2d_{n-2}(\alpha) .$$
    \item \label{dfactrec2}
     For $n\geq k \geq 1$, $d_n(\alpha) = d_k(\alpha) d_{n-k}(\alpha)-\alpha^2d_{k-1}(\alpha)d_{n-k-1}(\alpha)$.\\
    %\item \label{dfactsplitsquare}
    %For $n\geq 1$,  $d_n^2(\alpha) - d_{n-1}(\alpha)d_{n+1}(\alpha) = \alpha^{2n}$.\\ 
    %\item \label{dfact2l} 
    %For $k \geq 1$, $d_{2k}(\alpha) = d_k^2(\alpha)-\alpha^2d_{k-1}^2(\alpha)$, and for $k\geq2$,  ${d_{2k-1}(\alpha) = d_k^2(\alpha)-\alpha^4d_{k-2}(\alpha)^2}$. \\
    \item \label{dfactprod} 
    For $n\geq k\geq 1$, $d_k(\alpha)d_n(\alpha)-d_{k-1}(\alpha)d_{n+1}(\alpha) = \alpha^{2k}d_{n-k}(\alpha)$.\\
    \item \label{dfactboundseasy} 
    For $\alpha \in (0,.5)$ and $n\geq 1$, the sequence of functions $d_n(\alpha)$ satisfies the following bounds:  $$d_{n-1}(\alpha)>d_{n}(\alpha)>\frac12 d_{n-1}(\alpha)>0.$$
    % In fact, a stronger bound is also true: $d_{n-1}(\alpha)>d_{n}(\alpha)>\frac{1}{\sqrt{2}}  d_{n-1}(\alpha)>0$ for $\alpha \in (0,.5)$.\\
    \item \label{dfactatendpoints} 
    For $n\geq 0$,  $\displaystyle d_n\left(\frac 1 2\right) = \frac{n+1}{2^{n}}$ and %$\displaystyle d_n\left(\frac 1 {\sqrt{5}}\right) = \frac{\mathbb{F}_{n+1}}{5^{\lfloor \frac n 2 \rfloor}}$ 
    $\displaystyle d_n\left(\frac 1 {\sqrt{5}}\right) = \frac{(1+\phi)^n-\phi^n}{5^{ \frac n 2 }}$).     \\
    \item \label{dfactratio}
    % $\displaystyle \lim_{n \to \infty} \frac{d_{n+1}(\alpha)}{d_n(\alpha)} = \frac12(1+\sqrt{1-4\alpha^2})$. \\ 
    $\displaystyle \lim_{n \to \infty} \frac{d_{n+k}(\alpha)}{d_n(\alpha)} = \frac{(1+\sqrt{1-4\alpha^2})^k}{2^k}$. \\ 
    \item \label{dfactratioalpha}
    $\displaystyle \lim_{n \to \infty} \frac{\alpha^n}{d_n(\alpha)} = 0$.\\
    \item \label{dfactBF}
    Let 
    $\displaystyle \mathbb{FR}_n =
    \frac{\Phi^{n+1}-\varphi^{n+1}}{\sqrt{5}(\Phi^n - \varphi^n)}$,
    where $\varphi = \frac{\sqrt{5}-1}{2}$, and $\Phi = \frac{\sqrt{5}+1}{2}$.
    % \begin{cases}
    % \frac{L_{i+1}}{5F_i} & i \text{ even}\\[2mm]
    % \frac{F_{i+1}}{L_i} & i \text{ odd},
    % \end{cases} 
    %$
    %where $F_i$, and $L_i$ are the usual Fibonacci and Lucas numbers.
    Then,
    for $\alpha \in \left(0,\frac{1}{\sqrt{5}}\right)$ and $n\geq 1$, $$d_n \geq \mathbb{FR}_n \cdot d_{n-1}(\alpha).$$
\end{enumerate}

\end{theorem}

\subsection{Closed formulae for $K_{i,j}^n(\alpha)$ for the path and cycle graphs.}
We begin by determining a formula for the Katz similarity index between nodes $i$ and $j$ in a path of length $n$ as a function of $\alpha$. 
\begin{thm}\label{thm:Katzdist}
Let $P_n$ be the path graph on $n$ vertices, then 
\[
K_{P_n}(\alpha)_{i,j} = \frac{\alpha^{j-i} d_{i-1}(\alpha)d_{n-j}(\alpha)}{d_n(\alpha)}, 
\]
where $d_i$ is the function in Definition~\ref{def:dn}.
\end{thm}

 The proof of this result will require the following theorem, which originally gave the determinant of certain tridiagonal matrices.

     \begin{theorem}\label{thm:path-det}\cite[Equation 5]{thmrefHu}
If $G$ is a path graph on $n$ vertices, and $A$ its adjacency matrix, then 
\[
\det (I-\alpha A_{P_n}) = d_n(\alpha).%\sum_{m = 0}^{\lfloor{n/2}\rfloor} (-1)^m { n-m \choose m} \alpha^{2m}.
\]

\end{theorem}

\begin{proof}[Proof of Theorem \ref{thm:Katzdist}]

Let $M_{P_n} = (I - \alpha A_{P_n})$ for a path on $n$ nodes.  We can find the entries of $M_{P_n}^{-1}$ using Cramer's rule: 
\[
(M_{P_n})_{i,j}^{-1} =(-1)^{i+j}\frac{\det M_{P_n}(i,j)}{\det M_{P_n}}
\]
where $M_{P_n}(i,j)$ is the matrix obtained from $M_{P_n}$ by deleting the $i$th row and the $j$th column.
Consider the determinant of $M_{P_n}(i,j)$ and, without loss of generality, assume that $i \leq j$ since $M$ is symmetric. Then, 
\[
M_{P_n} = \begin{bmatrix} B & C\\D & E\end{bmatrix},
\]
where  
 $B$ is $(j-1)\times(j-1)$, $C$ is $(j-1) \times (n-j)$, $D$ is $(n-j) \times (j-1)$, and $E$ is $(n-j) \times (n-j)$. 
Recall the well-known formula for block matrix determinants, which states that, assuming block $E$ is invertible, 
\[
\det\begin{bmatrix} B & C\\D & E\end{bmatrix} = \det(B-CE^{-1}D)\text{det}(E).
\] 
Observe that  $E$ has $1$ on the diagonal and $-\alpha$ on the super/sub diagonal so it is, in fact, just $M_{P_{n-j}}$. Hence by Theorem~\ref{thm:path-det}, the determinant is equal to $d_{n-j}(\alpha)$, which is positive (by Theorem~\ref{thm:alldfacts} Part~\ref{dfactboundseasy}), and therefore invertible.  
Moreover we observe that the matrix $D$ is identically zero so 
\[
\det(M_{P_n}(i,j)) = \det(B-CE^{-1}D)\det(E) = \det(B)\det(E)=\det(B)d_{n-j}(\alpha).
\]  

We now consider the determinant of $B$.  
We again use the block matrix determinant formula this time setting 
\[
B=\begin{bmatrix} B' & C'\\D' & E'\end{bmatrix},
\]
where $B'$ is $(i-1)\times(i-1)$, $C'$ is $(i-1) \times (j-i)$, $D'$ is $(j-i) \times (i-1)$, and $E'$ is $(j-i) \times (j-i)$.  
We note that $E'$ is upper triangular with $-\alpha$ on the diagonal and hence is invertible with determinant equal to $(-\alpha)^{j-i}$.  
Similar to before, $D'$ is identically equal to zero hence 
\[
\det(B) = \det(B'-C'E'^{-1}D')\det(E') = \det(B')\det(E').
\] 
Since $B'$ has $1$ on the diagonal and $-\alpha$ on the super/sub diagonal, by Theorem~\ref{thm:path-det} it has determinant $d_{i-1}(\alpha)$.  
Putting everything together gives the desired result of 
\[
M_{P_n}(i,j) = (-\alpha)^{j-i}d_{i-1}(\alpha)d_{n-j}(\alpha).
\]
The formula for $(M_{P_n})_{{ij}}^{-1}$ follows by plugging the values into Cramer's rule.

\end{proof}

%\subsection{A closed formula for $K_{i,j}^n(\alpha)$ for a cycle on $n$ vertices.}
We now give an exact formula for Katz similarity index between pairs of vertices on a cycle. 
\begin{theorem}\label{thm:cycle-katz}
Let $C_n$ denote the cycle graph with $n>4$ vertices and $1\leq i< j \leq n$ and $0<\alpha<\frac{1}{\rho(A)}$, then
% \begin{equation}\label{eq:kcformula}
% K_{C_n}(\alpha)_{i,j} = \frac{N}{D} - \mathbf{1}_{ij}
% \end{equation}
% where
%   \[N =\alpha^{j-i}d_{n-j+i-1}(\alpha)+\alpha^{n-j+i}d_{j-i+1}(\alpha)\]
% and 
% \[D=d_{n-1}(\alpha) - 2\alpha^n -2\alpha^2 d_{n-2}(\alpha)\]

% and $ \mathbf{1}_{ij}$ is the Kronecker delta function, equal to $1$ if $i= j$ and $0$ otherwise.
\begin{equation}\label{eq:kcformula}
K_{C_n}(\alpha)_{i,j} = 
\frac{\alpha^{k}d_{n-k-1}(\alpha)+\alpha^{n-k}d_{k-1}(\alpha)}
{d_{n-1}(\alpha) - 2\alpha^n -2\alpha^2 d_{n-2}(\alpha)} 
- \mathbf{1}_{ij},
\end{equation}
where {$k = \min(j-i, n-j+i)$}, and $ \mathbf{1}_{ij}$ is the Kronecker delta function, equal to $1$ if $i= j$ and $0$ otherwise.

\end{theorem}
To prove this result, we will need the following lemma.  Since the proof is similar to that for the path, we present it in the Appendix.
\begin{lemma}
\label{lem:cycle-det}
For a cycle H with $n>4$ vertices with adjacency matrix $A_{C_n}$ and $0<\alpha<\frac{1}{\rho(A_{C_n})}$,

\[det(I-\alpha A_{C_n}) = d_{n-1}(\alpha) - 2\alpha^n -2\alpha^2 d_{n-2}(\alpha).\]
\end{lemma}

\begin{proof}
[Proof of Theorem~\ref{thm:cycle-katz}]

Let $M_{C_n} = I-\alpha A_{C_n}$, where $A_{C_n}$ is the adjacency matrix of the cycle graph on $n$ vertices.
We first observe that given a sequence of connected nodes, we can rearrange the nodes by designating a new initial node. 
Thus for $i\leq j$ where $r =\min(j-i,n-j+i)+1$, $K_{C_n}(\alpha)_{i,j} = K_{C_n}(\alpha)_{1,r}$.
Thus we only need to calculate $K_{C_n}(\alpha)_{1,r}$.

Using Cramer's rule, we can find $M_{C_n}^{-1}$: 
\[(M_{C_n})_{i,j}^{-1} = (-1)^{i+j}\frac{\det(M_{C_n}(i,j))}{\det(M_{C_n})}.\]

By Lemma~\ref{lem:cycle-det}, we know $\det(M_{C_n}) = d_{n-1}(\alpha) - 2\alpha^n -2\alpha^2 d_{n-2}(\alpha)$.
Next, we calculate $\det(M_{C_n}(1,r))$.
If $r=1$, $M_{C_n}(1,1) =I - \alpha A_{P_{n-1}}$, where $A_{P_{n-1}}$  is the adjacency matrix of the path graph on $n-1$ vertices,  so by Theorem~\ref{thm:path-det}, $\det( M_{C_n}(1,1)) = d_{n-1}(\alpha)$.

Now consider the case where $1<r \leq \lfloor n/2 \rfloor$. Let 
\[
M_{C_n}(1,r) = \begin{bmatrix} B & C \\ D & E\end{bmatrix},
\]
 where $B$ is  $(r-1)\times (r-1)$, $C$ is  $(r-1)\times (n-r)$, $D$ is $(n-r) \times (r-1)$, and $E$ is  $(n-r)\times (n-r)$. 
 Note that $E=I - \alpha A_{P_{n-r}}$, where $A_{P_{n-r}}$ is the adjacency matrix of the path graph on $n-r$ vertices, and thus is invertible. Then, 
 \[
 \det(M(1,r)) = \det(B-CE^{-1}D)\det(E)= \det(B-CE^{-1}D)d_{n-r}(\alpha).
 \]
 Let $\bar e_{i,j}$ be the $(i,j)$th entry of $E^{-1}$.
Since the only nonzero entry of $D$ is the $(n-r,1)$th entry and the only nonzero entries of $C$  is the $(r-1, 1)$st entry, both of which are $-\alpha$, it is straightforward to verify that the only nonzero entry of
$
CE^{-1}D$ is  the $(r-1,1)$th entry, which, by Theorem~\ref{thm:Katzdist}, is  
\[
\alpha^2\bar e_{1,n-r} =   \frac{\alpha^{n-r+1}}{d_{n-r}(\alpha)}.
\]
Then,  
\[B-CE^{-1}D=\begin{bmatrix}-\alpha   &1&-\alpha & 0 &\cdots  &0\\
 0 & -\alpha & 1&-\alpha &\ddots &0 \\ 
 \vdots &  \ddots& \ddots & \ddots&\ddots &\vdots\\
  0 &  \ddots& \ddots & \ddots &-\alpha &1\\
 -\frac{\alpha^{n-r-1}}{d_{n-r}(\alpha)} & 0&\cdots&\cdots&0&-\alpha\\  
 \end{bmatrix}, \]
Using cofactor expansion along the first column, we get
\[\det(B-CE^{-1}D) = (-\alpha)^{r-1}- (-1)^{r}\frac{\alpha^{n-r+1}}{d_{n-r}(\alpha)}d_{r-2}(\alpha).\]
Thus 
\[
\det(M_{C_n}(1,r)) =(-1)^{r-1 }(\alpha^{r-1}d_{n-r}(\alpha)+\alpha^{n-r+1}d_{r-2}(\alpha)), 
\]
and thus 
\[
K_{C_n}(\alpha)_{1,r} =(-1)^{1+r}(-1)^{r-1 }\frac{\alpha^{r-1}d_{n-r}(\alpha)+\alpha^{n-r+1}d_{r-2}(\alpha)}{d_{n-1}(\alpha) - 2\alpha^n -2\alpha^2 d_{n-2}(\alpha)}.
\]
Recalling that $k =r-1$ this yields
\[K_{C_n}(\alpha)_{i,j} =\frac{\alpha^{k}d_{n-k-1}(\alpha)+\alpha^{n-k}d_{k-1}(\alpha)}{d_{n-1}(\alpha) - 2\alpha^n -2\alpha^2 d_{n-2}(\alpha)}.
\]

\end{proof}

\section{Convergence of Katz similarity index in paths and cycles}\label{sec:Convergence}
Curiously, it seems that the Katz similarity index between two points is largely independent of the value of $n$.  Recall that the Katz similarity score  counts walks between vertices, with greater emphasis placed on shorter walks, i.e., values of $\alpha$, $K_{G}(\alpha)_{i,j}$ is increasingly a measure of only local behavior. 

\subsection{Convergence on paths}
The goal of this subsection is to show that $K_{P_n}(\alpha)_{i,j}$ converges as $n$ goes to infinity.
%As we see in Tables~\ref{tab:KPconv} and \ref{tab:KCconv},
%The rate at which $K_{G}(\alpha)_{i,j}$ converges is faster for paths than for cycles, and for lower values of $\alpha$.

\begin{pathconvth}
$K_{P_n}(\alpha)_{i,j}$ converges as $n$ goes to infinity. In fact, 
\[
\lim_{n \to \infty} K_{P_n}(\alpha)_{i,j}
=\alpha^{j-i}d_{i-1}(\alpha)\left(\frac{2}{1+\sqrt{1-4\alpha^2}}\right)^j
\]
\end{pathconvth}
\begin{proof}
From Theorem~\ref{thm:Katzdist}, we have that 
\[
 K_{P_n}(\alpha)_{i,j}
=\alpha^{j-i}d_{i-1}(\alpha)\frac{d_{n-j}(\alpha)}{d(\alpha)}.
\]
    The result follows by applying Theorem~\ref{thm:alldfacts} Part~\ref{dfactratio}.
\end{proof}

\subsection{Convergence on cycles}
In this subsection we show that $K_{C_n}(\alpha)_{i,j}$ converges as $n$ goes to infinity. We first introduce a helpful lemma, that allows us to express the denominator in \eqref{eq:kcformula} in a simplified form, we then show  convergence on cycles.
\begin{lemma}\label{lem:katD} Let $D_n(\alpha) = d_{n-1} - 2\alpha^n - 2\alpha^2d_{n-2}(\alpha),$ where $d_i(\alpha)$ is a defined in Definition~\ref{def:dn}.  Then $$D_{2\ell}(\alpha) = (1-4\alpha^2)d_{\ell-1}^2, \quad D_{2\ell+1} = (1-2\alpha)(\alpha^{2\ell} + (1+2\alpha)d_{\ell}d_{\ell-1}).$$\end{lemma}
\begin{proof}
First, we observe that 
\begin{align}\begin{split}\label{Dn1}
D_n(\alpha) &=d_{n-1}(\alpha) - 2\alpha^n -2\alpha^2 d_{n-2}(\alpha)\\
            &=d_{n-2}(\alpha)-\alpha^2 d_{n-3}(\alpha) - 2\alpha^n 
            -2\alpha^2 d_{n-3}(\alpha)+2\alpha^4 d_{n-4}(\alpha)\\
            &=d_{n-2}(\alpha) -2\alpha^{n-1} -2\alpha^2 d_{n-3}(\alpha)  
            -\alpha^2 (d_{n-3}(\alpha)-2\alpha^{n-2}-2\alpha^2 d_{n-4}(\alpha))- 4\alpha^n +2\alpha^{n-1}\\
            &=D_{n-1}(\alpha)  
            -\alpha^2 D_{n-2}(\alpha)+2\alpha^{n-1} (1- 2\alpha).
            \end{split}
\end{align}
We note that,
\[
D_2 = 1 - 2\alpha^2 -2\alpha^2 = (1-4\alpha^2) \cdot 1^2
\]
and
\[
D_3(\alpha)=1-\alpha^2 - 2\alpha^3 -2\alpha^2 = (1-2\alpha)(1+2\alpha+\alpha^2).
\]
Assume that $$D_{2\ell-2} =(1-4\alpha^2)d_{\ell-2}^2  \quad \text{and}\quad  D_{2\ell-1} = (1-2\alpha)(\alpha^{2\ell-2} + (1+2\alpha)d_{\ell-1}d_{\ell-2}).$$
Then, 
\begin{align*}
    D_{2\ell}(\alpha)&=D_{2\ell-1}(\alpha)
    -\alpha^2 D_{2\ell-2}(\alpha) + 2\alpha^{2\ell-1}(1-2\alpha)\\    
    &=(1-2\alpha)(\alpha^{2\ell-2} + (1+2\alpha)d_{\ell-1}d_{\ell-2})
    -\alpha^2 (1-4\alpha^2)d_{\ell-2}^2 + 2\alpha^{2\ell-1}(1-2\alpha)\\
    &=(1-2\alpha)\big(\alpha^{2\ell-2} + (1+2\alpha)d_{\ell-1}d_{\ell-2}
    -\alpha^2 (1+2\alpha)d_{\ell-2}^2 + 2\alpha^{2\ell-1}\big)\\
    &=(1-4\alpha^2)\big(\alpha^{2\ell-2} + (d_{\ell-1}d_{\ell-2}
    -\alpha^2 d_{\ell-2}^2) \big)=(1-4\alpha^2)\big(\alpha^{2\ell-2} + d_{\ell}d_{\ell-2}
    \big)
\end{align*}
By Theorem~\ref{thm:alldfacts} Part~\ref{dfactprod} (with $n = k = \ell-1$), $D_{2\ell}(\alpha) =(1-4\alpha^2)(d_{\ell-1}^2)  $. 

\noindent Next, assume that $$ D_{2\ell-1} = (1-2\alpha)(\alpha^{2\ell-2} + (1+2\alpha)d_{\ell-1}d_{\ell-2}) \quad \text{and}\quad D_{2\ell} =(1-4\alpha^2)d_{\ell-1}^2   $$
Then, by \eqref{Dn1}
\begin{align*}
    D_{2\ell+1}(\alpha)&=D_{2\ell} 
    -\alpha^2 D_{{2\ell}-1}+2\alpha^{2\ell}(1 - 2\alpha)\\
    &=(1-4\alpha^2)d_{\ell-1}^2 
    -\alpha^2 (1-2\alpha)(\alpha^{2\ell-2} + (1+2\alpha)d_{\ell-1}d_{\ell-2})+2\alpha^{2\ell}(1 - 2\alpha)\\
    &=(1 - 2\alpha)\big((1+2\alpha)d_{\ell-1}^2 
    -\alpha^2 (\alpha^{2\ell-2} + (1+2\alpha)d_{\ell-1}d_{\ell-2})+2\alpha^{2\ell}\big)\\
    &=(1 - 2\alpha)\big(\alpha^{2\ell}+ (1+2\alpha)(d_{\ell-1}^2 
    -\alpha^2 d_{\ell-1}d_{\ell-2})\big)=(1 - 2\alpha)\big(\alpha^{2\ell}+ (1+2\alpha) d_{\ell}d_{\ell-1}\big)\\
\end{align*}
\end{proof}

\begin{cycleconvth}
$K_{C_n}(\alpha)_{i,j}$ converges as $n$ goes to infinity.

\end{cycleconvth}

\begin{proof}
We show that as $n \to \infty$, $K_{C_n}(\alpha)_{i,j}$ converges to a function of $\alpha$ and $|j-i|$.  We begin by noting that without loss of generality due to the symmetry of the problem, we may assume that $i = 1$ and consider the equivalent problem for $K_{C_n}(\alpha)_{1,j-i+1}$.  We must consider 4 cases, depending on whether $n$ and $|j-i|$ are even or odd.  We detail the case of $|j-i|$ even and provide the limiting functions for the other cases as the proofs are very similar.  We write $d_{k}$ instead of $d_k(\alpha)$ for readability.

Let $n=2\ell$ and $|j-i| = 2m$, then by Theorem~\ref{thm:cycle-katz}, Theorem~\ref{thm:alldfacts} part~\ref{dfactrec2} and Lemma~\ref{lem:katD} we have
\begin{align*}
    K_{C_{}2\ell}(\alpha)_{1,1+2m} &=\frac{\alpha^{2m} d_{2\ell - 2m-1} +\alpha^{2\ell-2m}d_{2m-1}}{(1-4\alpha^2)d^2_{\ell -1}}\\
&=\frac{\alpha^{2m}}{1-4\alpha^2}\left[\frac{d^2_{\ell-m} - \alpha^4d_{\ell-m-2}^2}{d^2_{\ell-1}} + \frac{d_{2m-1}}{\alpha^{4m-2}}\left(\frac{\alpha^{\ell-1}}{d_{\ell-1}}\right)^2\right]\\
&=\frac{\alpha^{2m}}{1-4\alpha^2}\left[\left(\frac{d_{\ell-m}}{d_{\ell-1}}\right)^2 - \alpha^4\left(\frac{d_{\ell-m-2}}{d_{\ell-1}}\right)^2 + \frac{d_{2m-1}}{\alpha^{4m-2}}\left(\frac{\alpha^{\ell-1}}{d_{\ell-1}}\right)^2\right].\\
% &= \frac{\alpha^{2m}}{1-4\alpha^2}\left[\left(\frac{d_{\ell-m}}{d_{\ell-m+1}}\right)^2\left(\frac{d_{\ell-m+1}}{d_{\ell-m+2}}\right)^2\cdots \left(\frac{d_{\ell-2}}{d_{\ell-1}}\right)^2 -\right.\\
% &\qquad \left.\alpha^4\left(\frac{d_{\ell-m-2}}{d_{\ell-m-1}}\right)^2\left(\frac{d_{\ell-m-1}}{d_{\ell-m}}\right)^2\cdots\left(\frac{d_{\ell-2}}{d_{\ell-1}}\right)^2 + \frac{d_{2m+1}}{\alpha^{4m-2}}\left(\frac{\alpha^{\ell-1}}{d_{\ell-1}}\right)^2\right].
\end{align*}
By Theorem~\ref{thm:alldfacts} parts~\ref{dfactratio} and~\ref{dfactratioalpha}, taking the limit as $n\to \infty$ (i.e., $\ell\to \infty$)
of the prior line yields
\begin{align*}\lim_{n\to\infty} K_{C_{2\ell}}(\alpha)_{1,1+2m}&=\frac{\alpha^{2m}}{1-4\alpha^2}\left[\left(\frac{2}{1+\sqrt{1-4\alpha^2}}\right)^{2m-2}-\alpha^4\left(\frac{2}{1+\sqrt{1-4\alpha^2}}\right)^{2m+2}\right]\\
&=\frac{\alpha^{2m}}{1-4\alpha^2}\left(\frac{2}{1+\sqrt{1-4\alpha^2}}\right)^{2m-2}\left[1-\alpha^4\left(\frac{2}{1+\sqrt{1-4\alpha^2}}\right)^{4}\right]\\
%&=\frac{\alpha^{2m}}{1-4\alpha^2}\left(\frac{2}{1+\sqrt{1-4\alpha^2}}\right)^{2m-2}\left[\frac{-32\alpha^2+8+8\sqrt{1-4\alpha^2}-16\alpha^2\sqrt{1-4\alpha^2}}{\left(1+\sqrt{1-4\alpha^2}\right)^{4}}\right].
\end{align*}
Now assume $n=2\ell+1$ and $|j-i| = 2m$, then by Theorem~\ref{thm:cycle-katz} and Lemma~\ref{lem:katD} we have
\begin{align*}
    K_{C_{2\ell+1}}(\alpha)_{1,1+2m} &=\frac{\alpha^{2m} d_{2\ell - 2m} +\alpha^{2\ell-2m+1}d_{2m-1}}{(1-2\alpha)(\alpha^{2\ell}+(1+2\alpha)d_{\ell}d_{\ell -1})}\\
    &=\frac{\alpha^{2m}}{1-2\alpha}\left(\frac{d_{2\ell-2m+1}}{\alpha^{2\ell}+(1+2\alpha)d_{\ell}d_{\ell -1}}+\frac{\alpha^{2\ell-4m+1}d_{2m-1}}{\alpha^{2\ell}+(1+2\alpha)d_{\ell}d_{\ell -1}}\right).
   \end{align*}
Consider
 \begin{align*}
 A_\ell &=\frac{d_{2\ell-2m}}{\alpha^{2\ell}+(1+2\alpha)d_{\ell}d_{\ell -1}}, \quad B_{\ell} = \frac{\alpha^{2\ell-4m+1}d_{2m-1}}{\alpha^{2\ell}+(1+2\alpha)d_{\ell}d_{\ell -1}}.
 \end{align*}
 Note that
 \begin{align*}
 |B_\ell | <\frac{\alpha^{2\ell-4m+1}d_{2m-1}}{ (1+2\alpha)d_{\ell}d_{\ell-1}}
 = \frac{d_{2m-1}}{\alpha^{4m-2}(1+2\alpha)}\frac{\alpha^{2\ell-1}}{ d_{\ell}d_{\ell-1}} \to 0
 \end{align*}

 On the other hand, by Theorem~\ref{thm:alldfacts} part~\ref{dfactratio}
  \begin{align*}
\lim_{n\to \infty} \frac{d_{2l-2m}}{d_{\ell}d_{\ell-1}}&=\lim_{n\to \infty}\frac{d_{\ell-m}^2-\alpha^2d^2_{\ell-m-1}}{d_{\ell}d_{\ell-1}}\\
 &=\lim_{n\to \infty}\left(\frac{d_{\ell-m}}{d_{\ell-1}}\right)^2\left(\frac{d_{\ell-1}}{d_{\ell}}\right) -\alpha^2\left(\frac{d_{\ell-m-1}}{d_{\ell-1}}\right)^2\left(\frac{d_{\ell-1}}{d_{\ell}}\right)\\
 &=\left[\left(\frac{2}{1+\sqrt{1-4\alpha^2}}\right)^{2m-1}-\alpha^2\left(\frac{2}{1+\sqrt{1-4\alpha^2}}\right)^{2m+1}\right].
\end{align*}
 So, 
 \begin{align*}
 \lim_{n\to \infty}\frac{1}{A_{\ell}}&= \lim_{n\to\infty}\left[\frac{\alpha^{2\ell}}{d_{2\ell-2m}}+(1+2\alpha)\frac{d_{\ell}d_{\ell-1}}{d_{2\ell-2m}}\right]\\&=(1+2\alpha)\left[\left(\frac{2}{1+\sqrt{1-4\alpha^2}}\right)^{2m-1}-\alpha^2\left(\frac{2}{1+\sqrt{1-4\alpha^2}}\right)^{2m+1}\right]^{-1}
 \end{align*}

Thus,
\begin{align*}\lim_{n\to\infty} K_{C_{2\ell+1}}(\alpha)_{1,1+2m}&=\frac{\alpha^{2m}}{1-4\alpha^2}\left[\left(\frac{2}{1+\sqrt{1-4\alpha^2}}\right)^{2m-1}-\alpha^2\left(\frac{2}{1+\sqrt{1-4\alpha^2}}\right)^{2m+1}\right]\\
&=\frac{\alpha^{2m}}{1-4\alpha^2}\left(\frac{2}{1+\sqrt{1-4\alpha^2}}\right)^{2m-2}\left[\frac{2}{1+\sqrt{1-4\alpha^2}}-\alpha^2\left(\frac{2}{1+\sqrt{1-4\alpha^2}}\right)^{3}\right]
%&=\frac{\alpha^{2m}}{1-4\alpha^2}\left(\frac{2}{1+\sqrt{1-4\alpha^2}}\right)^{2m-2}\left[\frac{-32\alpha^2+8+8\sqrt{1-4\alpha^2}-16\alpha^2\sqrt{1-4\alpha^2}}{\left(1+\sqrt{1-4\alpha^2}\right)^{4}}\right].
\end{align*}
It is an easy exercise to show that 
\[
1-\alpha^4\left(\frac{2}{1+\sqrt{1-4\alpha^2}}\right)^{4}
=\frac{2}{1+\sqrt{1-4\alpha^2}}-\alpha^2\left(\frac{2}{1+\sqrt{1-4\alpha^2}}\right)^{3}.
\]
The proof proceeds similarly for the case that $|j-i|$ is odd, yielding \\[2mm]
\begin{equation*}
\lim_{n\to\infty} K_{C_n}(\alpha)_{i,j}=
\begin{cases}
\displaystyle \frac{\alpha^{2m}}{1-4\alpha^2}\left(\frac{2}{1+\sqrt{1-4\alpha^2}}\right)^{2m-2}\left[1-\alpha^4\left(\frac{2}{1+\sqrt{1-4\alpha^2}}\right)^{4}\right],& \text{if $|j-i| = 2m$}\\[5mm]
\displaystyle  \frac{\alpha^{2m-1}}{1-4\alpha^2}\left(\frac{2}{1+\sqrt{1-4\alpha^2}}\right)^{2m-3}\left[1-\alpha^4\left(\frac{2}{1+\sqrt{1-4\alpha^2}}\right)^{4}\right],&  \text{if $|j-i| = 2m+1$}.\\
\end{cases}
\end{equation*}

\end{proof}
     
\section{Relationship between Distance, Resistance Distance, and Katz similarity index}\label{sec:relationship}

 It is a natural open problem to determine the values of $\alpha$ that will result in the same predicted link ordering for the Katz similarity index and the resistance distance in various families of graphs. Along these lines,  we present Theorems~\ref{thm:cycleordering} and \ref{conj:Katzordering}.  Unlike in prior sections, we start with the cycle graph and then move to the path graph.   

\subsection{Cycles}
In cycle graphs all distance $k$ pairs of vertices give the same Katz similarity index and resistance distance. Thus, our result relating the three distances can be stated without restriction to location on the cycle.  
\begin{ktzcycle}
Let $C_n$ be a cycle graph on $n$ vertices, with associated adjacency matrix $A_n$. For every valid value of $\alpha$ (i.e., $0<\alpha < 1/\rho(A_n)$)  the effective resistance, distance, and Katz similarity rank all node pairs of $G$ in the same order.  That is, 
\[
K_{C_n}(\alpha)_{i,j} < K_{C_n}(\alpha)_{i',j'}
%\text{ if and only if }
\Leftrightarrow
R_{C_n}(i,j)>R_{C_n}(i',j') 
%\text{ if and only if }
\Leftrightarrow
d_{C_n}(i,j)>d_{C_n}(i',j')
\]
for all $i,i',j,j' \in [n]$.

%    Let $A_n$ be the adjacency matrix for the cycle graph on $n$ vertices, and suppose $\alpha < 1/2 = 1/\rho(A_n)$. Then $K_{C_n}(i_1,i_1+k)(\alpha) > K_{C_n}(i_2,i_2+k+1)(\alpha)$ if and only if $d_{C_n}(i_1,i_1 + k)<d_{C_n}(i_2,i_2+k+1)$ if and only if $R_{C_n}(i_1,i_1+k)<R_{C_n}(i_2,i_2+k+1)$.
\end{ktzcycle}
\begin{proof}
Without loss of generality assume that $j > i$ and that $k = j-i < \lfloor{n/2}\rfloor.$
Note first that for  $K_{C_n}(i,i+k) = K_{C_n}(1,1+k)$,  $d_{C_n}(i,i + k)=d_{C_n}(1,1 + k)$, and  $R_{C_n}(i,i+k)=R_{C_n}(1,1+k)$ for any $i$ and $k$. Thus, without loss of generality, we can proceed by assuming that $i=i'=1$. 
For $n$ odd and  $k = \lfloor n/2 \rfloor$, $n-k = k+1$, so $K_{C_n}(1,1+k) =K_{C_n}(1,2+k) $, $d_{C_n}(1,1 + k)=d_{C_n}(1,2 + k)$, and $R_{C_n}(1,1+k)=R_{C_n}(1,2+k)$.

It is straightforward to verify that  $d_{C_n}(1,1 + k)<d_{C_n}(1,1+k+1)$ if and only if $R_{C_n}(1,1+k)<R_{C_n}(1,1+k+1)$ if and only if $k < \lfloor \frac n2 \rfloor$.  For the last equivalence ($K_{C_n}(1,1+k)>K_{C_n}(1,1+k+1)$), we consider the difference of the associated numerators (as defined in Theorem~\ref{thm:cycle-katz}:

    \begin{equation*}
        %K_{C_n}(1,1+k)(\alpha) - K_{C_n}(1,2+k)(\alpha)\\
        \Delta_{n,k}=\alpha^kd_{n-k-1}(\alpha)+\alpha^{n-k}d_{k-1}(\alpha)-\alpha^{k+1}d_{n-k-2}(\alpha)-\alpha^{n-k-1}d_{k}(\alpha),
    \end{equation*}
    First we show that for $n$ even and $k = \frac n 2 -1$, $\Delta_{n,k}>0$. 
    Set $\ell = \frac n2$ and $k = \ell-1$:
    \begin{align*}
        \Delta_{n,k}&= \alpha^{\ell-1}d_{\ell}(\alpha)+\alpha^{\ell+1}d_{\ell-2}(\alpha)-\alpha^{\ell}d_{\ell-1}(\alpha)-\alpha^{\ell}d_{\ell-1}(\alpha)\\
        &= \alpha^{\ell-1}[d_{\ell-1}-\alpha^2d_{\ell-2}]+\alpha^{\ell+1}d_{\ell-2}(\alpha)-2\alpha^{\ell}d_{\ell-1}(\alpha)\\
        &=\alpha^{\ell-1}d_{\ell-1}(1-2\alpha^2)>0.
     \end{align*}
 Next, we induct (backwards!) on $k$. 
 Note that for $1\leq k < \lfloor \frac n 2\rfloor$
    \begin{align*}
        \Delta_{n,k}&=\alpha^kd_{n-k-2}(\alpha)-\alpha^{k+2}d_{n-k-3}(\alpha) -\alpha^{k+1}d_{n-k-2}(\alpha)\\
        &\qquad +\alpha^{n-k-2}(d_{k}(\alpha)-d_{k+1}(\alpha))-\alpha^{n-k-1}d_{k}(\alpha)\\
        &=\alpha^kd_{n-k-2}(\alpha)(1-\alpha)-\alpha^{k+2}d_{n-k-3}(\alpha)\\
        &\qquad +\alpha^{n-k-2}d_{k}(\alpha)(1-\alpha)-\alpha^{n-k-2}d_{k}(\alpha)d_{k+1}(\alpha)\\
        &>\alpha^{k+1}d_{n-k-2}(\alpha)-\alpha^{k+2}d_{n-k-3}(\alpha)\\
        &\qquad +\alpha^{n-k-1}d_{k}(\alpha)-\alpha^{n-k-2}d_{k}(\alpha)d_{k+1}(\alpha)\\
        &= \Delta_{n,k+1}.
    \end{align*}
Thus, whenever $\Delta_{n,k}>0$ and $k>1$, then $\Delta_{n,k-}>0$.

    \end{proof}

\subsection{Paths}
In this section, our aim is to prove Theorem~\ref{conj:Katzordering}. A useful first step is to identify distance-$k$  pairs of vertices  that have the smallest and largest Katz similarity index.  %To show this we utilize  Proposition~\ref{prop:Apowers}.

\begin{prop}\label{prop:monotone}
For $\alpha \in (0,.5)$, $K_{P_n}(\alpha)_{i,i+k} \leq K_{P_n}(\alpha)_{i+1,i+k+1}$ where $1\leq i,k\leq n$ are positive integers and $0 \leq n - k - 2i -1 $. 

\end{prop}

\begin{proof}
Recall from Theorem~\ref{thm:Katzdist} that 
\[
K_{P_n}(\alpha)_{i+1,i+k+1}-K_{P_n}(\alpha)_{i,i+k} =
\frac{\alpha^{k}(d_i(\alpha)d_{n-k-i-1}(\alpha)-d_{i-1}(\alpha)d_{n-k-i}(\alpha))}{d_n(\alpha)}.
\]
From Theorem~\ref{thm:alldfacts} Parts~\ref{dfactprod} and \ref{dfactboundseasy} we  have that
\[
d_i(\alpha)d_{n-k-i-1}(\alpha)-d_{i-1}(\alpha)d_{n-k-i}(\alpha) = \alpha^{2i}d_{n-k-2i-1}(\alpha)>0.
\]

\end{proof}

Consider the polynomial
\[
{}_np_j(\alpha) = \min_i K_{P_n}(\alpha)_{i,i+j} - \max_i K_{P_n}(\alpha)_{i,i+j+1} %= \min_i \sum_{k=1}^\infty (A_{P_n}^k)_{i,i+j} \alpha^k - \max_i \sum_{k=1}^\infty (A^k_{P_n})_{i,i+j+1}\alpha^k
\]
Note that any values of $\alpha$ for which ${}_np_j(\alpha)$ is positive correspond to values of $\alpha$ for which the Katz similarity index is lower for all distance-$j$ pairs of vertices than for any distance-$(j+1)$ pairs. 
%Theorem~\ref{thm:monotone} implies that $\min_i \sum_{k=1}^\infty (A_{P_n}^k)_{i,i+j} x^k$ occurs at $i=1$, and $\max_i \sum_{k=1}^\infty (A_{P_n}^k)_{i,i+j+1} x^k$ occurs at $i=m = \lfloor \frac n 2 \rfloor$, thus 
Proposition~\ref{prop:monotone} implies that $\min_i K_{P_n}(\alpha)_{i,i+j}$ occurs at $i=1$, and $ \max_i K_{P_n}(\alpha)_{i,i+j+1}$ occurs at $i=m = \lceil \frac{n-j} 2 \rceil$, thus 

\[
{}_np_j(\alpha)% = \sum_{k=1}^\infty (A_{P_n}^k)_{1,j+1}\alpha^k - \sum_{k=1}^\infty (A_{P_n}^k)_{m,m+j+1} \alpha^k 
=  K_{P_n}(\alpha)_{1,j+1} - K_{P_n}(\alpha)_{m,m+j+1}.
\]

Theorem~\ref{thm:Katzdist} and Theorem \ref{thm:alldfacts} Part \ref{dfactboundseasy} give the following fact, which reduces the proof of Theorem~\ref{conj:Katzordering} to a simpler case. 
\begin{prop}\label{prop:reduceThm2}
    Let 
    ${}_n\tilde p_j(\alpha) = d_{n-j-1}(\alpha)-\alpha d_{m-1}(\alpha)d_{n-m-j-1}(\alpha)$.
    Then, 
    \[
    {}_n\tilde p_j(\alpha) =0 \iff {}_n p_j(\alpha) =0 \quad \text{and} \quad 
    {}_n\tilde p_j(\alpha) >0 \iff {}_n p_j(\alpha) >0.
    \]
\end{prop}

We now recall and prove Theorem~\ref{conj:Katzordering}:
\begin{ktzordering} Let $P_n$ be a path graph on $n$ of vertices. The resistance distance and usual graph distance are equal, thus they always rank all node pairs equivalently.  If $\alpha < 1/{\sqrt{5}}$, then effective resistance (equivalently, distance) and Katz similarity rank all node pairs of $G$ in the same order. That is, 
\[
K_{P_n}(\alpha)_{i,j} < K_{P_n}(\alpha)_{i',j'}
\text{ if and only if }
R_{P_n}(i,j)>R_{P_n}(i',j')
\]
for $\alpha< 1/\sqrt{5}$ and for all $i,i',j,j' \in [n]$.
\end{ktzordering}

\begin{proof}[Proof of Theorem \ref{conj:Katzordering}]
We use Theorem~\ref{thm:alldfacts} Parts \ref{dfactrec2} and~\ref{dfactBF}.
\begin{align*}
    {}_n\tilde{p}_j(\alpha) &= d_{n-j-1}(\alpha)-\alpha d_{m-1}(\alpha)d_{n-m-j-1}(\alpha)\\
    &=d_{m-1}(\alpha)(d_{n-j-m}(\alpha) - \alpha d_{n-j-m-1}(\alpha) - \alpha^2d_{m-2}(\alpha)d_{n-j-m-1}(\alpha)\\
    &\geq d_{m-1}(\alpha)(d_{n-j-m-1}(\alpha)(\mathbb{FR}_{n-j-m}-\alpha)-\alpha^2d_{m-2}(\alpha)d_{n-j-m-1}(\alpha)\\
    &\geq d_{n-j-m-1}(\alpha)d_{m-2}(\alpha)(\mathbb{FR}_{m-1}(\mathbb{FR}_{n-j-m}-\alpha)-\alpha^2)\\
    &=d_{n-j-m-1}(\alpha)d_{m-2}(\alpha)\left[\frac{((1+\phi)^m-\phi^m)(\mathbb{FR}_{n-j-m}-\alpha) - \alpha^2\sqrt{5}((1+\phi)^{m-1}-\phi^m)}{\sqrt{5}((1+\phi)^{m-1}-\phi^{m-1})}\right]\\
    &=\frac{d_{n-j-m-1}(\alpha)d_{m-2}(\alpha)}{5((1+\phi)^{m-1}-\phi^m)((1+\phi)^{n-m-j}-\phi^{n-m-j}}\cdot\\
    &\qquad \left[ ((1+\phi)^m-\phi^m)((1+\phi)^{n-m-j+1}-\phi^{n-m-j+1} - \alpha \sqrt{5}((1+\phi)^{n-m-j}-\phi^{n-m-j})\right.\\
    &\qquad\qquad \left.-5\alpha^2((1+\phi)^{m-1}-\phi^{m-1})((1+\phi)^{n-m-j}-\phi^{n-m-j})\right]\\
        &\geq\frac{d_{n-j-m-1}(\alpha)d_{m-2}(\alpha)}{5((1+\phi)^{m-1}-\phi^m)((1+\phi)^{n-m-j}-\phi^{n-m-j}}\cdot\\
    &\qquad \left[ ((1+\phi)^m-\phi^m)((1+\phi)^{n-m-j+1}-\phi^{n-m-j+1} - (1+\phi)^{n-m-j}-\phi^{n-m-j}\right.\\
    &\qquad\qquad \left.-((1+\phi)^{m-1}-\phi^{m-1})((1+\phi)^{n-m-j}-\phi^{n-m-j})\right]
\end{align*}
Where the last line is due to the fact that $\alpha < 1/\sqrt{5}.$  From the line above, it is clear that the fraction is positive since $(1+\phi) >\phi.$  We now consider just the part in the square brackets.  Rearranging terms we find that it is equal to 
\begin{multline*}(1+\phi)^{n-j-1}((1+\phi)^2-(1+\phi) -1) + \phi^{n-j-1}(\phi^2-\phi-1) +\\ \phi^{m-1}(1+\phi)^{n-j-m}(-\phi(1+\phi) + \phi +1) + (1+\phi)^{m-1}\phi^{n-j-m}(-\phi(1+\phi)+1+\phi+1).\end{multline*}
A straight forward calculation shows that
\begin{align*}
(1+\phi)^2-(1+\phi) -1 &= 0\\
\phi^2-\phi-1 &= -2\phi\\
-\phi(1+\phi) + \phi +1 &= \phi\\
-\phi(1+\phi) + 1+\phi +1 &= 1+\phi.
\end{align*}
Thus the term in square brackets simplifies (using the fact, that $(1+\phi) = \phi^{-1}$) to
\begin{equation}\label{eq:presplit}
  -2\phi^{n-j+1}+\frac{\phi^m}{\phi^{n-j-m}}+\frac{\phi^{n-j-m}}{\phi^m}.  
\end{equation}
If $n-j$ is even then \eqref{eq:presplit} simplifies to
\[2(1-\phi^{n-j+1}) > 0,\]
and if $n-j$ is odd then \eqref{eq:presplit} simplifies to
\[-2\phi^{n-j+1}+2+2\phi+ > 0.\]

\end{proof}

\begin{prop}\label{prop:Katroot5}
% Let  $m  = \lfloor n/2 \rfloor$. Then for  $n\geq 1+i$
% \[
% d_n\left(\frac{1}{\sqrt{5}}\right)K_{1,1+i}^n\Big(\frac{1}{\sqrt{5}}\Big) = 
%   \begin{cases}
%     \frac{{F}_{n-i}
% }{\sqrt{5}^{(n-2)}} & n-i \text{ even}\\[2mm]
%     \frac{L_{n-i}}{\sqrt{5}^{(n-1)}} & n-i \text{ odd}\\
%     \end{cases}, 
%     \]
%     and
%     \[
% d_n\left(\frac{1}{\sqrt{5}}\right)K_{m,m+i+1}^n\Big(\frac{1}{\sqrt{5}}\Big) = 
% \begin{cases}
%     \frac{L_{n-i}-2}{\sqrt{5}^{(n-1)}} & n-i \text{ even}\\[2mm]
%     \frac{F_{n-i}-1}{\sqrt{5}^{(n-2)}} & n-i \text{ odd}\\
% \end{cases}
% \]
 
% Thus,   
    \[
    {}_n\tilde p_j \left(\frac{1}{\sqrt{5}}\right)
     =\frac{(1+\phi)^{2m+j-n}+\phi^{2m+j-n}-2\phi^{n-j}}
{\sqrt{5}^{ n-j-1}}.
    % \begin{cases}
    %     \frac{\sqrt{5}F_{n-j}-(L_{n-j}-2)}{\sqrt{5}^{(n-1)}} & n -j\text{ even}\\[2mm]
    %     \frac{L_{n-j}-\sqrt{5}(F_{n-j}-1)}{\sqrt{5}^{(n-1)}} & n -j\text{ odd}\\
    %     \end{cases}
    \]
\end{prop}
The proof of Proposition~\ref{prop:Katroot5} is a straightforward application of Theorem~\ref{thm:alldfacts} Part~\ref{dfactatendpoints}, and is included in the Appendix. 

    \begin{corollary}\label{cor:numroot5}
    \[
    {}_n\tilde p_j\Big(\frac{1}{\sqrt{5}}\Big)>0 \text{ for all $n$, and }
    \lim_{n\to \infty} {}_n\tilde p_j\Big(\frac{1}{\sqrt{5}}\Big)  = 0.
    \]    \end{corollary}
    \begin{proof}
    For the inequality, first observe that the denominator is always positive and that the numerator is equivalent to 
    \[[(1+\phi)^{2m+j-n}-\phi^{n-j}]+\phi^{2m+j-n}(1-\phi^{2n-2j-2m}).\]
    Since $(1+\phi) > 1$ the first term in the square bracket is always greater than 1, and since $\phi < 1 $ the second term in the square bracket is always less then 1, hence the square bracket is always positive.  Moreover since $2n -2j-2m > 1$ the second term in the sum is also always positive, hence $\displaystyle {}_n\tilde p_j\Big(\frac{1}{\sqrt{5}}\Big)>0 .$  For the limit we consider each term separately
    \[\lim_{n\to\infty}\frac{(1+\phi)^{2m+j -n}}{\sqrt{5}^{ n-j-1}} = \lim_{n\to\infty} \begin{cases}\frac{1}{\sqrt{5}^{-j-1}}\left(\frac{1}{\sqrt{5}}\right)^n & \text{if $n-j$ is even}\\ \frac{(1+\phi)}{\sqrt{5}^{-j-1}}\left(\frac{1}{\sqrt{5}}\right)^n & \text{if $n-j$ is odd}\end{cases} = 0\]
    \[\lim_{n\to\infty}\frac{\phi^{2m+j -n}}{\sqrt{5}^{ n-j-1}} = \lim_{n\to\infty} \begin{cases}\frac{1}{\sqrt{5}^{-j-1}}\left(\frac{1}{\sqrt{5}}\right)^n & \text{if $n-j$ is even}\\ \frac{\phi}{\sqrt{5}^{-j-1}}\left(\frac{1}{\sqrt{5}}\right)^n & \text{if $n-j$ is odd}\end{cases} = 0\]
    \[\lim_{n\to\infty}-2\frac{\phi^{n-j}}{\sqrt{5}^{ n-j-1}} = \lim_{n\to\infty} -2\frac{\phi^{-j}}{\sqrt{5}^{-j-1}}\left(\frac{\phi}{\sqrt{5}}\right)^n  = 0.\]
    % Recalling the 
    % the following well-known identity (\red{see Vajda-50c, I Ruggles (1963) FQ 1.2 pg 80, Rabinowitz-25, B\&Q(2003)-Identity 243})
    % \begin{equation}\label{id:negphi}
    % (-\phi)^n =\frac{L_n-\sqrt{5}F_n}{2}
    % \end{equation}we have that 
    %     \[
    % {}_n\tilde p_i \left(\frac{1}{\sqrt{5}}\right)
    % = 
    % \begin{cases}
    %     \frac{-2(-\phi)^{n-i} +2}{\sqrt{5}^{(n-1)}} & n -i\text{ even}\\[2mm]
    %     \frac{2(-\phi)^{n-i}+\sqrt{5}}{\sqrt{5}^{(n-1)}} & n -i\text{ odd}\\
    %     \end{cases}
    % \]
    % for $\phi = \frac{\sqrt{5}-1}{2} \approx 0.618$

    \end{proof}
A direct application of Theorem~\ref{thm:alldfacts} Part~\ref{dfactatendpoints} to the equation for ${}_n\tilde{p}_j$ given in Proposition~\ref{prop:reduceThm2} yields the following. 
    \begin{lemma}\label{lem:numhalf}

    %Let  $m  = \lfloor (n-i+1)/2 \rfloor$.
%     The numerators of $\displaystyle K_{1,2}^n\Big(\frac{1}{2}\Big)$ and $\displaystyle K_{m,m+2}^n\Big(\frac{1}{2}\Big)$ are
%    \[
%    d_n\left(\frac{1}{2}\right)K_{1,1+i}^n\Big(\frac{1}{2}\Big) = \frac{n-i}{2^{n-1}}
% , \quad \text{for } n\geq 1+i
% \]
% and \[
%    d_n\left(\frac{1}{2}\right)K_{m,m+1+i}^n\Big(\frac{1}{2}\Big) = 
%        \frac{(n-i-m)m}{2^{n-1}}, \quad \text{for } n\geq 2+i
%      \]
     
    % The numerator of $\displaystyle K_{1,2}^n\Big(\frac{1}{2}\Big)-K_{m,m+2}^n\Big(\frac{1}{2}\Big)$ is given by  
    % \[
    % d_n\left(\frac{1}{2}\right)\left(K_{1,2}^n\left(\frac{1}{2}\right)-K_{m,m+2}^n\left(\frac{1}{2}\right) \right)
    % = 
    %    \frac{n-i-m(n-i-m)}{2^{n-1}} 
    % \]

    \[
    {}_n\tilde p_j\Big(\frac12\Big) = \frac{n-j-m(n-m-j)}{2^{n-j-1}}=\frac{n-j-\lceil \frac {n-j} 2\rceil(\lfloor \frac {n-j} 2\rfloor -j)}{2^{n-j-1}}\]
    and thus
    \[
            {}_n\tilde p_j\Big(\frac12\Big)=0 \text{ for }  n-j=4, \quad \text{and} \quad
            {}_n\tilde p_j\Big(\frac12\Big)<0 \text{ for }  n-j>4.
    \]
    \end{lemma}
    %\begin{proof}
    %    \red{Yes, it really is just plug and chug.  DO we even need a proof?}
    %\end{proof}
Lemma~\ref{lem:numhalf} and Corollary~\ref{cor:numroot5} give the following result supporting Conjecture \ref{conj:Katzroots}.
\begin{thm}\label{thm:rootsexist}
    For $n-j\geq  5$, ${}_np_j(\alpha)$ has a root between $\frac{1}{\sqrt{5}}$ and $\frac12$. 
\end{thm}

\begin{figure}[tph]
     \centering
      \begin{subfigure}[b]{0.45\textwidth}
         \centering
	   \includegraphics[width=.95\linewidth]{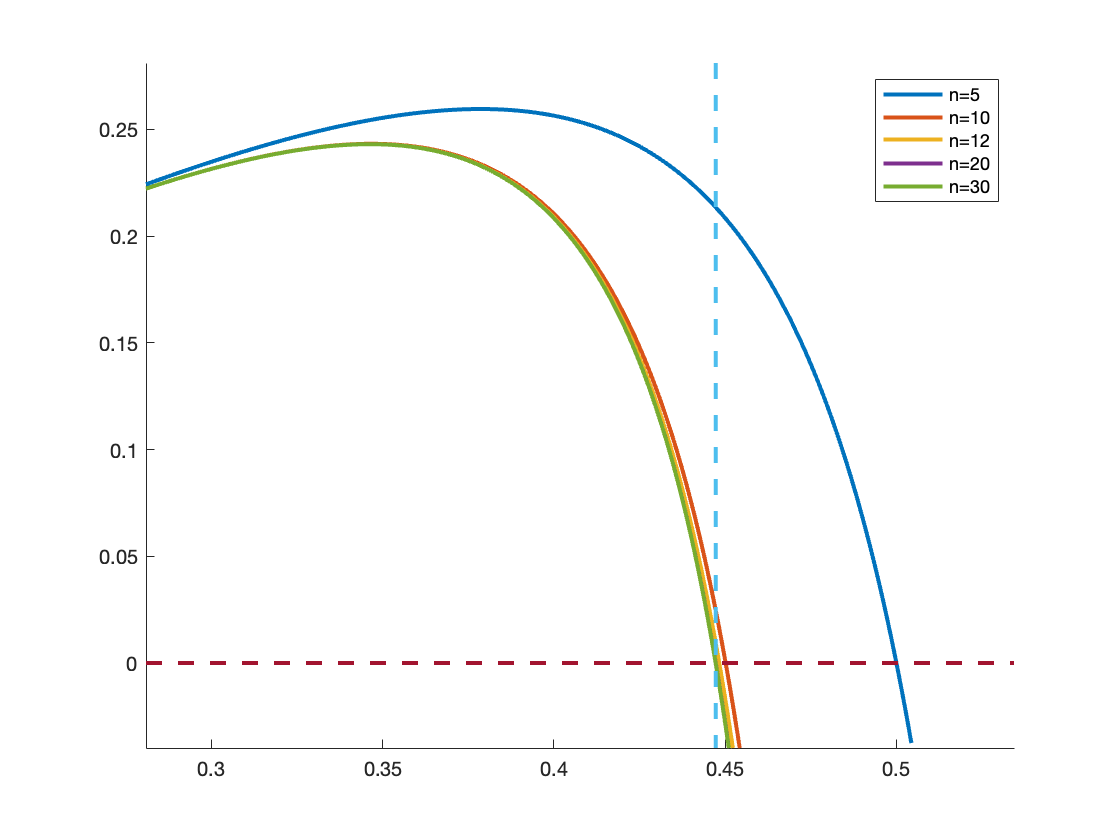}
     \end{subfigure}
     \begin{subfigure}[b]{0.45\textwidth}
         \centering
	   \includegraphics[width=.95\linewidth]{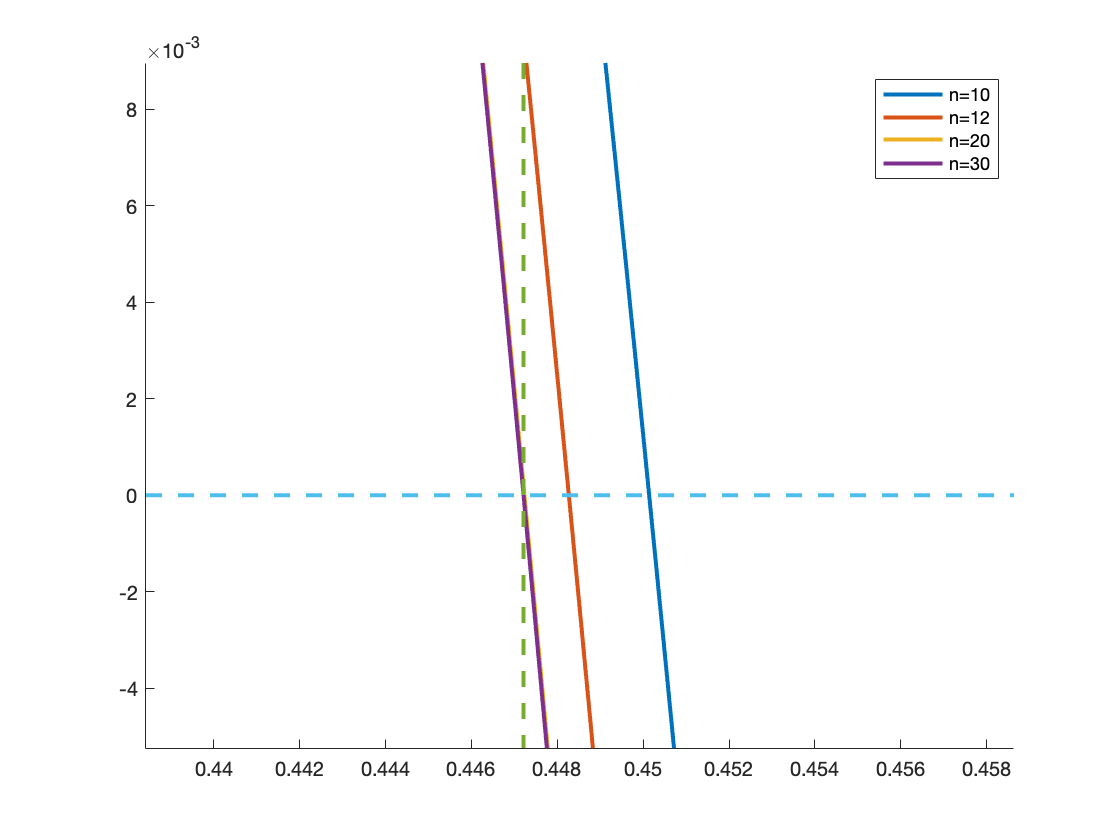}
     \end{subfigure}
     \caption[Behavior near $1/\sqrt{5}$]{In these plots we consider the behavior of $p_1$ near $\alpha = 1/\sqrt{5}$ for several values of $n$.  The $x$-axis corresponds to $\alpha$ values and the $y$-axis to values of ${}_np_1$. The plot on the right is a magnified version of the one on the left.  The value $\alpha = 1/\sqrt{5}$ is indicated by a dashed line. We note that even under high magnification there is essentially no difference between $n= 20$ and $n =30$, which provides further evidence of the convergence of these functions as $n \to \infty$.}
         \label{fig:remainder}
    \end{figure}

We call the roots $\alpha_n$ of the polynomials ${}_n p_1$ the {\it cut-off} values for the path graph $P_n$. For any $\alpha$ above such a cut-off value, the Katz similarity index and resistance distance rankings agree on all pairs of vertices and disagree for some pair of vertices if $\alpha>\alpha_n$. In Figure~\ref{fig:remainder} we plot the polynomials ${}_np_1$ near $\alpha = 1/\sqrt{5}$ for several values of $n$.   Positive $y$-values correspond to agreement between Katz similarity index and resistance distance rankings on pairs of vertices.  We see that the values of the cut-off values, and indeed the behavior of the ${}_np_1$ appear indistinguishable quickly.

\begin{cnjroots}
The cut-off values $\alpha_n$ converge to $\frac{1}{\sqrt{5}}$ from above. 
\end{cnjroots}
The investigation of these cut-off points is an interesting open problem that has not yet been investigated in the literature.

\section{Appendix}

\begin{proprootfive}
    \[
    {}_n\tilde p_j \left(\frac{1}{\sqrt{5}}\right)
     =\frac{(1+\phi)^{2m+j-n}+\phi^{2m+j-n}-2\phi^{n-j}}
{\sqrt{5}^{ n-j-1}}
    \]
    \end{proprootfive}
    \begin{proof}
By Theorem~\ref{thm:alldfacts} Part \ref{dfactatendpoints}, $d_{i}\Big(\frac{1}{\sqrt{5}}\Big) =\frac{(1+\phi)^{i+1} - \phi^{i+1}}{\sqrt{5}^{i }}$.
Then, 
\begin{align*}
{}_n\tilde p_j \left(\frac{1}{\sqrt{5}}\right)
&= d_{n-j-1}\Big(\frac{1}{\sqrt{5}}\Big) -\frac{1}{\sqrt{5}} \cdot  d_{m-1}\Big(\frac{1}{\sqrt{5}}\Big)d_{n-m-j-1}\Big(\frac{1}{\sqrt{5}}\Big) \\
&=\frac{(1+\phi)^{n-j} - \phi^{n-j}}{\sqrt{5}^{ n-j-1}} - \frac{1}{\sqrt{5}}\cdot \frac{(1+\phi)^{m} - \phi^{m}}{\sqrt{5}^{ m-1}} \cdot \frac{(1+\phi)^{n-m-j} - \phi^{n-m-j}}{\sqrt{5}^{ n-m-j-1}}\\[2mm]
&=\frac{(1+\phi)^{n-j} - \phi^{n-j}-((1+\phi)^{m} - \phi^{m})((1+\phi)^{n-m-j} - \phi^{n-m-j})}{\sqrt{5}^{ n-j-1}}\\[2mm]
&=\frac{(1+\phi)^{n-j} - \phi^{n-j}-
(1+\phi)^{n-j}-\phi^{n-j}+ (1+\phi)^{2m+j-n}+\phi^{2m+j-n}}
{\sqrt{5}^{ n-j-1}}\\[2mm]
&=\frac{(1+\phi)^{2m+j-n}+\phi^{2m+j-n}-2\phi^{n-j}}
{\sqrt{5}^{ n-j-1}}
\end{align*}

\end{proof}

\begin{lemmacycle1}

For a cycle H with $n>4$ vertices with adjacency matrix $A_{C_n}$ and $0<\alpha<\frac{1}{\rho(A_{C_n})}$,

\[\det(I-\alpha A_{C_n}) = d_{n-1}(\alpha) - 2\alpha^n -2\alpha^2 d_{n-2}(\alpha).\]
\end{lemmacycle1}

\begin{proof}

Let $M_{C_n} =(I-\alpha A_{C_n})$. 
For this graph, we will use cofactor expansion along the first column to calculate $\det(M_{C_n})$ and define $M_{C_n}(i,j)$ to be $M_{C_n}$ with the $i$th row and $j$th column removed. 

Expanding, along the first column we find

\[\det(M_{C_n}) =  M_{C_n}(1,1)+\alpha M_{C_n}(2,1) + (-1)^{n+1}(-\alpha)M_{C_n}(n,1).\]

First, we observe that $M_{C_n}(1,1) = M_{P_{n-1}}$ and then we determine $\det(M_{C_n}(2,1) )$ and $\det(M_{C_n}(n,1))$ using cofactor expansion along the first row.

Hence
\begin{align*}
\det(M_{C_n}(2,1)) =& -\alpha M_{P_{n-2}}  + (-1)^{1+n-1}(-\alpha) T_{n-2}\\
=& -\alpha d_{n-2}(\alpha) + (-1)^n (-\alpha)(-\alpha)^{n-2}\\
%=& -\alpha d_{n-2}(\alpha) + (-1)^n (-\alpha)^{n-1}\\
=& -\alpha d_{n-2}(\alpha) -\alpha^{n-1}
\end{align*}
and

\begin{align*}
\det(M_{C_n}(n,1)) =& (-\alpha) L_{n-2} + (-1)^{1+n-1}(-\alpha) M_{P_{n-2}}\\
=& (-\alpha)(-\alpha)^{n-2} +(-1)^{1+n-1}(-\alpha) d_{n-2}\\
=& (-\alpha)^{n-1} -(-1)^n\alpha d_{n-2}(\alpha),
\end{align*}
where $T_{n-2}$ is an upper triangular matrix with $-\alpha$ along the diagonal and $L_{n-2}$ is a lower triangular matrix with $-\alpha$ along the diagonal.

Plugging the cofactors in, we get
\begin{align*}
\det(M_{C_n}) =& d_{n-1}(\alpha) +\alpha \bigg(-\alpha d_{n-2}(\alpha) -(\alpha)^{n-1} \bigg)+ (-1)^{n+1}(-\alpha)\bigg((-\alpha)^{n-1} -(-1)^n\alpha d_{n-2}(\alpha)\bigg)\\
%=&d_{n-1}(\alpha)-\alpha^2d_{n-2}(\alpha)-\alpha^n - \alpha^n- \alpha^2 d_{n-2}(\alpha)\\
=& d_{n-1}(\alpha)-2\alpha^n - 2\alpha^2d_{n-2}(\alpha).
\end{align*}

\end{proof}

\begin{proof}[Proof of Theorem~\ref{thm:alldfacts}]
\ 

\vspace*{3mm} 

\noindent \textit{\textbf{Part  \ref{dfactrecursion}.} 
$d_0(\alpha) = d_{1}(\alpha) = 1$, and for $i\geq 2$, $ d_n(\alpha) = d_{n-1}(\alpha) - \alpha^2d_{n-2}(\alpha) $}.
    \begin{align*}
     d_{i}(\alpha)  &= 
     \sum_{k=0}^{\lfloor \frac i 2\rfloor} (-1)^k {i -k \choose k} \alpha^{2k}
     = 1 + \sum_{k=1}^{\lfloor \frac i 2\rfloor-1} (-1)^k {i -k \choose k} \alpha^{2k}+  (-1)^{\lfloor \frac i 2\rfloor }{i - \lfloor \frac i 2 \rfloor \choose \lfloor \frac i 2\rfloor}\alpha^{2\lfloor \frac i 2\rfloor}\\
      &= 1 + \sum_{k=1}^{\lfloor \frac i 2\rfloor-1} (-1)^k {i-1 -k \choose k} \alpha^{2k}+\sum_{k=1}^{\lfloor \frac i 2\rfloor-1} (-1)^k {i-1 -k \choose k-1} \alpha^{2k} + (-1)^{\lfloor \frac i 2\rfloor}{i - \lfloor \frac i 2 \rfloor \choose \lfloor \frac i 2\rfloor} \alpha^{2\lfloor \frac i 2\rfloor}\\
      & = \sum_{k=0}^{\lfloor \frac i 2\rfloor-1} (-1)^k {i-1 -k \choose k} \alpha^{2k}-\alpha^2\sum_{k=0}^{\lfloor \frac {i-2} 2\rfloor-1} (-1)^k {i-2 -k \choose k} \alpha^{2k} + (-1)^{\lfloor \frac i 2\rfloor}{i - \lfloor \frac i 2 \rfloor \choose \lfloor \frac i 2\rfloor} \alpha^{2\lfloor \frac i 2\rfloor}
     \end{align*}
Note that 
\[
 \sum_{k=0}^{{\lfloor \frac i 2\rfloor}-1} (-1)^k {i-1 -k \choose k} \alpha^{2k} = 
 \begin{cases}
     \displaystyle\sum_{k=0}^{{\lfloor \frac {i-1} 2\rfloor}} (-1)^k {i-1 -k \choose k} \alpha^{2k}  & i \text{ even}\\[4mm]
     \displaystyle\sum_{k=0}^{{\lfloor \frac {i-1} 2\rfloor}} (-1)^k {i-1 -k \choose k} \alpha^{2k}  - (-1)^{\lfloor \frac i 2\rfloor}  \alpha^{2{\lfloor \frac i 2\rfloor}} & i \text{ odd}
 \end{cases}
\]
and
\[
\sum_{k=0}^{\lfloor \frac {i-2} 2\rfloor-1} (-1)^k {i-2 -k \choose k} \alpha^{2k}
=\sum_{k=0}^{\lfloor \frac {i-2} 2\rfloor} (-1)^k {i-2 -k \choose k} \alpha^{2k}+ (-1)^{\lfloor \frac {i} 2\rfloor}  {i-1 -\lfloor \frac {i} 2\rfloor \choose \lfloor \frac i 2 \rfloor -1} \alpha^{2\lfloor \frac {i} 2\rfloor-2}.
\]
The result follows from 
$\displaystyle
{i - \lfloor \frac i 2 \rfloor \choose \lfloor \frac i 2\rfloor} -  {i-1 -\lfloor \frac {i} 2\rfloor \choose \lfloor \frac i 2 \rfloor -1}  = 
\begin{cases}
0 & i \text{ even}\\
1 & i \text{ odd}.    \\[2mm]
\end{cases}
$

\noindent \textit{\textbf{Part  \ref{dfactrec2}.}
For $n\geq k \geq 1$, $d_n(\alpha) = d_k(\alpha) d_{n-k}(\alpha)-\alpha^2d_{k-1}(\alpha)d_{n-k-1}(\alpha)$.
}\\ [1mm]
Part~\ref{dfactrecursion} gives 
\[
d_n(\alpha)= d_{1}(\alpha)d_{n-1}(\alpha)-\alpha^2d_{0}(\alpha)d_{n-2}(\alpha).
\]
Assume $d_{n}(\alpha) = d_{k-1}(\alpha)d_{n-k+1}(\alpha)-\alpha^2 d_{k-2}(\alpha)d_{n-k}(\alpha) $. Then (again, using Part~\ref{dfactrecursion}):
\begin{align*}
d_{n}(\alpha) &= d_{n-k}(\alpha)(d_{k-1}(\alpha)-\alpha^2 d_{k-2}(\alpha ) )- \alpha^2 d_{k-1}(\alpha)d_{n-k-1}(\alpha)\\
&= d_{n-k}(\alpha)d_{k}(\alpha) - \alpha^2 d_{k-1}(\alpha)d_{n-k-1}(\alpha).\\
\end{align*}

\noindent \textit{\textbf{Part~\ref{dfactprod}:} For $n\geq k\geq 1$, $d_k(\alpha)d_n(\alpha)-d_{k-1}(\alpha)d_{n+1}(\alpha) = \alpha^{2k}d_{n-k}(\alpha)$.}\\[2mm]
Let $k\geq1$. Then, 
\[
d_1d_k-d_0d_{k+1} = d_k-d_{k+1} = \alpha^2d_{k-1}
\]
Assume that $k\geq i\geq 2$ and $d_{i-1}d_{k-1}-d_{i-2}d_{k} = \alpha^{2(i-1)}d_{i-k}$. Then, 
\begin{align*}
d_id_k-d_{i-1}d_{k+1} &= d_{i-1}d_k-\alpha^2d_{i-2}d_k-d_{i-1}d_{k}+\alpha^2d_{i-1}d_{k-1} \\
&=  -\alpha^2d_{i-2}d_k+\alpha^2d_{i-1}d_{k-1} = \alpha^2\alpha^{2(i-1)}d_{i-k}.
\end{align*}

\noindent\textit{\textbf{Part  \ref{dfactboundseasy}.} For $n\geq 2$,   $d_{n-1}(\alpha)>d_{n}(\alpha)>\frac{1}{2}  d_{n-1}(\alpha)>0$.}\\[1mm]
    
     Note that for $\alpha \in (0,.5)$
     $$
     d_1(\alpha)= 1>d_2(\alpha) = 1-\alpha^2 >\frac 34 >\frac 1{2}\cdot d_1(\alpha) >0.
     $$     
      Assume that $d_{n-2}(\alpha)> d_{n-1}(\alpha)>\frac12 d_{n-2}(\alpha)>0$. Then $\frac12 d_{n-1}(\alpha)>0$ by hypothesis, and since $d_{n-2}(\alpha)>0$ also by hypothesis,  $d_{n-1}(\alpha) > d_{n}(\alpha)$ follows from Part~\ref{dfactrecursion}. The remaining inequality follows from 
       \[
      d_{n}(\alpha) - \frac12 d_{n-1}(\alpha) = \frac12 d_{n-1} - \alpha^2 d_{n-2}(\alpha) >\left(\frac14 - \alpha^2\right)d_{n-2}(\alpha)>0.
      \]

\noindent\textit{\textbf{Part  \ref{dfactatendpoints}.} $\displaystyle d_n\left(\frac 1 2\right) = \frac{n+1}{2^{n}}$ and $\displaystyle d_n\left(\frac 1 {\sqrt{5}}\right) =  \frac{(1+\phi)^{n+1}-\phi^{n+1}}{\sqrt{5}
^{n }}$.}\\[2mm]

    By Part \ref{dfactrecursion}, $d_0(\frac 1 2) = 1 = \frac 1 {2^0}$, $d_1(\frac 1 2) = 1 = \frac 2 {2^1}$. Assume $d_{n-1} (\frac 1 2)= \frac n {2^{n-1}}$ and $d_{n-2}(\frac 1 2) = \frac{n-1}{2^{n-2}}$, then 
    \begin{align*}
    d_n\Big(\frac 1 2\Big)&=d_{n-1}\Big(\frac 1 2\Big)-\frac14 \cdot d_{n-2}\Big(\frac 1 2\Big) \\
    &= \frac n {2^{n-1}}-\frac 1 4 \cdot  \frac{n-1}{2^{n-2}}=  \frac{2n - n+1}{2^{n}}=  \frac{ n+1}{2^{n}}
    \end{align*}

By Part \ref{dfactrecursion}, 
$$
d_0\Big(\frac 1 {\sqrt{5}}\Big) = 1 = \frac {(1+\phi)^{1}-\phi^{1}}{\sqrt{5}
^{0 }}
\quad
\text{and} 
\quad 
d_1\Big(\frac 1 {\sqrt{5}}\Big)  = 1 = \frac {(1+\phi)^{2}-\phi^{2}}{\sqrt{5}
^{1 }}.
$$
Assume that $d_{n-1} \big(\frac 1 {\sqrt{5}}\big)= \frac{(1+\phi)^{n}-\phi^{n}}{\sqrt{5}
^{n-1 }}$ and $d_{n-2}(\frac 1 {\sqrt{5}}) = \frac{(1+\phi)^{n-1}-\phi^{n-1}}{\sqrt{5}
^{n-2 }}$.
Then     
\begin{align*}
    d_n\left(\frac 1 {\sqrt{5}}\right) &=
    d_{n-1}\left(\frac 1 {\sqrt{5}}\right)- \frac15 \cdot 
    d_{n-2}\left(\frac 1 {\sqrt{5}}\right)\\
& = \frac{(1+\phi)^{n}-\phi^{n}}{\sqrt{5}
^{n-1 }}-\frac15 \cdot\frac{(1+\phi)^{n-1}-\phi^{n-1}}{\sqrt{5}
^{n-2 }}\\
& = \frac{\sqrt{5}((1+\phi)^{n}-\phi^{n})-(1+\phi)^{n-1}+\phi^{n-1}}{\sqrt{5}
^{n }}\\
& = \frac{(1+\phi)^{n-1}(\sqrt{5}(1+\phi) -1) - \phi^{n-1}(\sqrt{5}\phi - 1)}{\sqrt{5}
^{n }}.
\end{align*}
    It is easy to check that 
    \[
    \sqrt{5}(1+\phi) -1 = (1+\phi)^2\quad \text{and}
    \quad \sqrt{5}\phi - 1 = \phi^2.
    \]
 \noindent\textit{\textbf{Part  \ref{dfactratio}.} $\displaystyle \lim_{n \to \infty} \frac{d_{n+k}(\alpha)}{d_n(\alpha)} = \frac{(1+\sqrt{1-4\alpha^2})^k}{2^k}$.}\\[2mm]
 First, we show that $\displaystyle \lim_{n \to \infty} \frac{d_{n+1}(\alpha)}{d_n(\alpha)} = \frac12(1+\sqrt{1-4\alpha^2})$. Note that this result is true if and only if 
$$\lim_{n\to\infty}(2d_{n+1}-d_n)^2= d_n^2(1-4\alpha^2)$$ if and only if
$$\lim_{n\to \infty} d_{n+1}^2-d_{n+1}d_n + \alpha^2d_n^2 = 0.$$
However
\begin{align*}
d_{n+1}^2-d_{n+1}d_n + \alpha^2d_n^2 &= d_{n+2}d_n+\alpha^{2n+2}-d_{n+1}d_n + \alpha^2d_n^2 \\
&=\alpha^{2n+2} + d_n(d_{n+2}-d_{n+1})+\alpha^2d_n^2\\
&= \alpha^{2n+2} + d_n(d_{n+2}-d_{n+1})+d_n(d_{n+1}-d_{n+2})\\
&= \alpha^{2n+2}.
\end{align*}
Since $\alpha\in[0,1/2)$, $\alpha^{2n+2} \to 0$ as $n\to \infty.$

Next, note that 
\[
\displaystyle  \frac{d_{n+k}(\alpha)}{d_n(\alpha)} = \displaystyle  \frac{d_{n+1}(\alpha)}{d_n(\alpha)}\frac{d_{n+2}(\alpha)}{d_{n+1}(\alpha)}\cdots \frac{d_{n+k}(\alpha)}{d_{n+k-1}(\alpha)}, 
\]
so
\[
\lim_{n \to \infty}\frac{d_{n+k}(\alpha)}{d_n(\alpha)} = \prod_{i = 1}^k \left(\lim_{n \to \infty}\frac{d_{n+i}(\alpha)}{d_{n+i-1}(\alpha)}\right) = \frac{(1+\sqrt{1-4\alpha^2})^k}{2^k}.
\]

\noindent\textit{\textbf{Part  \ref{dfactratioalpha}.} $\displaystyle \lim_{n \to \infty} \frac{\alpha^n}{d_n(\alpha)} = 0$.} \\[2mm]
We show that $\displaystyle \frac{\alpha^n}{d_n} \leq \frac{2\alpha}{(n+1)}$ by showing that 
\[
d_n(\alpha)\geq  \frac{(n+1)}{2n}d_{n-1}(\alpha) \geq  \ldots > \frac{n+1}{2n}\frac{n}{2(n-1)}\cdots \frac{3}{2} d_{1}(\alpha) = \frac{n+1}{2^{n-1}}\geq  (n+1)\alpha^{n-1}>\frac{(n+1)}{2}\alpha^{n-1}. 
\]
To obtain the above we only need to show that $d_n(\alpha) \geq \frac{(n+1)}{2n} d_{n-1}$. Note that for $\alpha \in (0,.5)$, 
$$
d_1(\alpha)= 1,
$$ 
and
\[d_2(\alpha) = 1-\alpha^2 > \frac 34 = \frac{2+1}{2\cdot 2}d_1(\alpha).\]
Assume that $ d_{n-1}(\alpha)\geq \frac{n}{2(n-1)}\, d_{n-2}(\alpha)$. Then, 
 \begin{align*}
     d_{n}(\alpha) - \frac{n+1}{2n} d_{n-1}(\alpha) &= \frac{n-1}{2n} d_{n-1} - \alpha^2 d_{n-2}(\alpha)\\
     &\geq \frac{n-1}{2n}\cdot\frac{n}{2(n-1)}\, d_{n-2}(\alpha) - \alpha^2 d_{n-2}(\alpha) \\
     &=\left(\frac14 - \alpha^2\right)d_{n-2}(\alpha)\geq 0.
 \end{align*}

So we have that $d_n(\alpha) \geq \frac{(n+1)}{2n} d_{n-1}$.\\[3mm]

\noindent \textit{\textbf{Part~\ref{dfactBF}:} Let $\mathbb{FR}_n = \frac{\Phi^{n+1} - \phi^{n+1}}{\sqrt{5}(\Phi^n - \phi^n)}$. For $\alpha \in \left(0,\frac{1}{\sqrt{5}}\right)$ and $n\geq 1$, $d_n \geq \mathbb{FR}_n \cdot d_{n-1}(\alpha)$.}\\[2mm]

Since $\mathbb{FR}_1 = \frac{\Phi^{2} - \phi^{2}}{\sqrt{5}(\Phi - \phi)}  = \frac{\Phi + \phi}{\sqrt{5}}=1$,
$ d_1 = 1 \geq \mathbb{FR}_n\cdot  d_0  = 1 $.

Assume $d_{n-1} \geq \mathbb{FR}_{n-1} \cdot d_{n-2}$. Then, 
\begin{align*}
d_n - \mathbb{FR}_n \cdot d_{n-1} &= d_{n-1} - \alpha^2 d_{n-2}- \mathbb{FR}_n \cdot d_{n-1} \\
&= d_{n-1}(1-\mathbb{FR}_n) - \alpha^2 d_{n-2} \\
&\geq d_{n-2}\mathbb{FR}_{n-1}(1-\mathbb{FR}_n) - \alpha^2 d_{n-2} \\
&= d_{n-2}(\mathbb{FR}_{n-1}(1-\mathbb{FR}_n) - \alpha^2 )
\end{align*}
It is an easy exercise to verify that 
 $\mathbb{FR}_{n-1}(1-\mathbb{FR}_n) = \frac15$, and therefore 
 \[
 \mathbb{FR}_{n-1}(1-\mathbb{FR}_n) - \alpha^2>0 \text{ for }\alpha \in \left(0,\frac{1}{\sqrt{5}}\right).
 \]

  \end{proof}

\bibliographystyle{plain}
\bibliography{references}
\end{document}